\documentclass[12pt]{article}
\usepackage{amsthm,amsmath,amsfonts,amssymb}
\usepackage{graphicx}
\graphicspath{ {Graphics/} }
\usepackage{algorithm}
\usepackage{algpseudocode}

\makeatletter
\def\BState{\State\hskip-\ALG@thistlm}
\makeatother

\title{Semi-Lagrangian one-step methods for two classes of time-dependent partial differential systems \vspace{-2ex}}
\author{}
\date{}

\newtheorem{theorem}{Theorem}
\newtheorem{definition}{Definition}
\newtheorem{lemma}{Lemma}

\begin{document}
\maketitle

\centerline{\scshape Nikolai D. Lipscomb and Daniel X. Guo}
\medskip
{\footnotesize
\centerline{Department of Mathematics and Statistics}
\centerline{University of North Carolina Wilmington}
\centerline{Wilmington, North Carolina, USA}
}

\abstract{Semi-Lagrangian methods are numerical methods designed to find approximate solutions to particular time-dependent partial differential equations (PDEs) that describe the advection process. We propose semi-Lagrangian one-step methods for numerically solving initial value problems for two general systems of partial differential equations. Along the characteristic lines of the PDEs, we use ordinary differential equation (ODE) numerical methods to solve the PDEs. The main benefit of our methods is the efficient achievement of high order local truncation error through the use of Runge-Kutta methods along the characteristics. In addition, we investigate the numerical analysis of semi-Lagrangian methods applied to systems of PDEs: stability, convergence, and maximum error bounds.}

\vspace{5mm}

\section{Introduction}

Time-dependent partial differential equations are at the core of particle physics. Due to the difficult and often analytically unsolvable nature of most of these equations, the next best approach is to use an algorithm to approximate the solution. Research on numerical computation for approximating solutions to advection equations goes back to the 1950s with finite difference method approaches to nonlinear hyperbolic partial differential equations by Courant et al. \cite{rcourant} and numerical integration of the barotropic vorticity equation by Fj{\o}rtoft \cite{rfjortoft}. In the fields of weather forecasting and climate modelling, particle trajectory methods were proposed by Wiin-Nielsen \cite{wiin-nielsen} which led to the only reliable forecasting model of its time. Meanwhile, researchers in plasma physics saw the promise of semi-Lagrangian approaches; for example, Cheng and Knorr \cite{chengknorr} produced an efficient numerical splitting scheme for solving the Vlasov-Maxwell equations. While there was much research on advection processes and weather prediction over the following decades, the 1980s produced a slew of research that brought characteristic-based methods into different numerical approaches: Douglas Jr. and Russell \cite{jdouglas} brought the method of characteristics to finite difference and finite element methods, Andr{\'e} Robert's meterological contributions produced stable numerical solutions to the shallow-water equations \cite{robert}, and many more. Today, semi-Lagrangian models are frequently used by organisations that focus on atmospheric modelling such as the European Centre for Medium-Range Weather Forecasts (ECMWF) \cite{diamantakis}, the National Oceanic and Atmospheric Administration (NOAA), the National Center for Atmospheric Research \cite{lauritzen}, and the High Resolution Local Area Modelling (HIRLAM) programme.

The popularity of semi-Lagrangian methods in today's atmospheric models lies in the resolution problem. Eulerian schemes' accuracy is dependent on the resolution of the solution grid, a function of the problem domain's discretisation--specifically the temporal and spatial discretisation. For earlier Eulerian schemes, resolution was greatly dependent on stability, which demanded a very small time discretisation relative to the spatial discretisation \cite{staniforthcote}. The gradual development of methods to avoid such constraints brought attention to semi-Lagrangian schemes: a pairing of the equal spacing of solutions from an Eulerian approach and the particle-tracing of a Lagrangian approach. Modern semi-Lagrangian numerical schemes perform with great numerical stability under a wide range of resolutions and produce little numerical dispersion \cite{diamantakis}. In order to maintain competitiveness in the near future, semi-Lagrangian-dependent numerical schemes must continue to improve: reduction of error while maintaining sufficiently fast computation time considers not just the resolution of the Eulerian grid, but the order of the error. Most numerical schemes in weather applications achieve second order results with respect to the spatial and time discretisations. Further, while developed from physical laws, a complete numerical analysis of semi-Lagrangian theory is still underway.

We will examine numerical methods for solving initial value problems involving systems of time-dependent partial differential equations (PDEs). The particular types of PDEs we will examine are PDEs that describe the advection process, mostly found in fluid dynamics and atmospheric modelling. We will consider two general cases for systems of time-dependent PDEs.
The first system, \emph{general advection in one dimension}, is
\begin{equation} \left\{
\begin{array}{cccc}
\dfrac{\partial y_1}{\partial t} + \omega \dfrac{\partial y_1}{\partial x} & = & f_1(t,x,y_1,\dots,y_n), & \ \ y_{1_0}=y_1(0,x), \\
\vdots & \vdots & \vdots & \vdots \\
\dfrac{\partial y_n}{\partial t} + \omega \dfrac{\partial y_n}{\partial x} & = & f_n(t,x,y_1,\dots,y_n), & \ \ y_{n_0}=y_n(0,x),
\end{array} \right.
\label{mainsystem1}
\end{equation}
where $y_i = y_i(t,x)$, $(t,x) \in [0,\infty) \times [a,b]$ and $\omega$ belongs to one of three cases:
\begin{enumerate}
\item $\omega \in \mathbb{R}$,
\item $\omega = \omega(t, x)$, $| \omega | < \infty$,
\item $\omega = \omega(t, x, \boldsymbol{y})$, $| \omega | < \infty$.
\end{enumerate}

The second system, \emph{nonlinear advection in two dimensions}, is of the form
\begin{equation} \left\{
\begin{array}{cccc}
\dfrac{\partial u}{\partial t} + u \dfrac{\partial u}{\partial x} + v \dfrac{\partial u}{\partial y} & = & f(t,x,y,u,v), & \ \ u_0=u(0,x,y), \\
&&& \\
\dfrac{\partial v}{\partial t} + u \dfrac{\partial v}{\partial x} + v \dfrac{\partial v}{\partial y} & = & g(t,x,y,u,v), & \ \ v_0=u(0,x,y), \\
\end{array} \right.
\label{mainsystem2}
\end{equation}
where $u = u(t,x,y)$, $v = v(t,x,y)$, $(t,x,y) \in [0,\infty) \times [a,b] \times [c,d]$.

An Eulerian scheme examines a prescribed set of points and examines how the solution to the PDEs change at these points as time goes on. Since these time-dependent PDEs model physical processes, such as the movement of particles, a Lagrangian approach would examine individual particles (or parcels) with respect to the solution and trace their trajectory, examining how the solution updates at the new arrival points as time goes on. A semi-Lagrangian method is a marriage of the two concepts: we preserve an Eulerian framework by constructing a grid that keeps the analysis of the solution spread evenly throughout a region of interest; however, we also examine the parcels that pass through these grid points, tracing their trajectory and using that information to update the solution at later times. By considering the system of ordinary differential equations (ODEs) along the characteristic lines of the PDE, we can take advantage of numerical methods that solve ODEs \cite{guo1, guo2}.

\begin{figure}
\centering
\includegraphics[width=125mm]{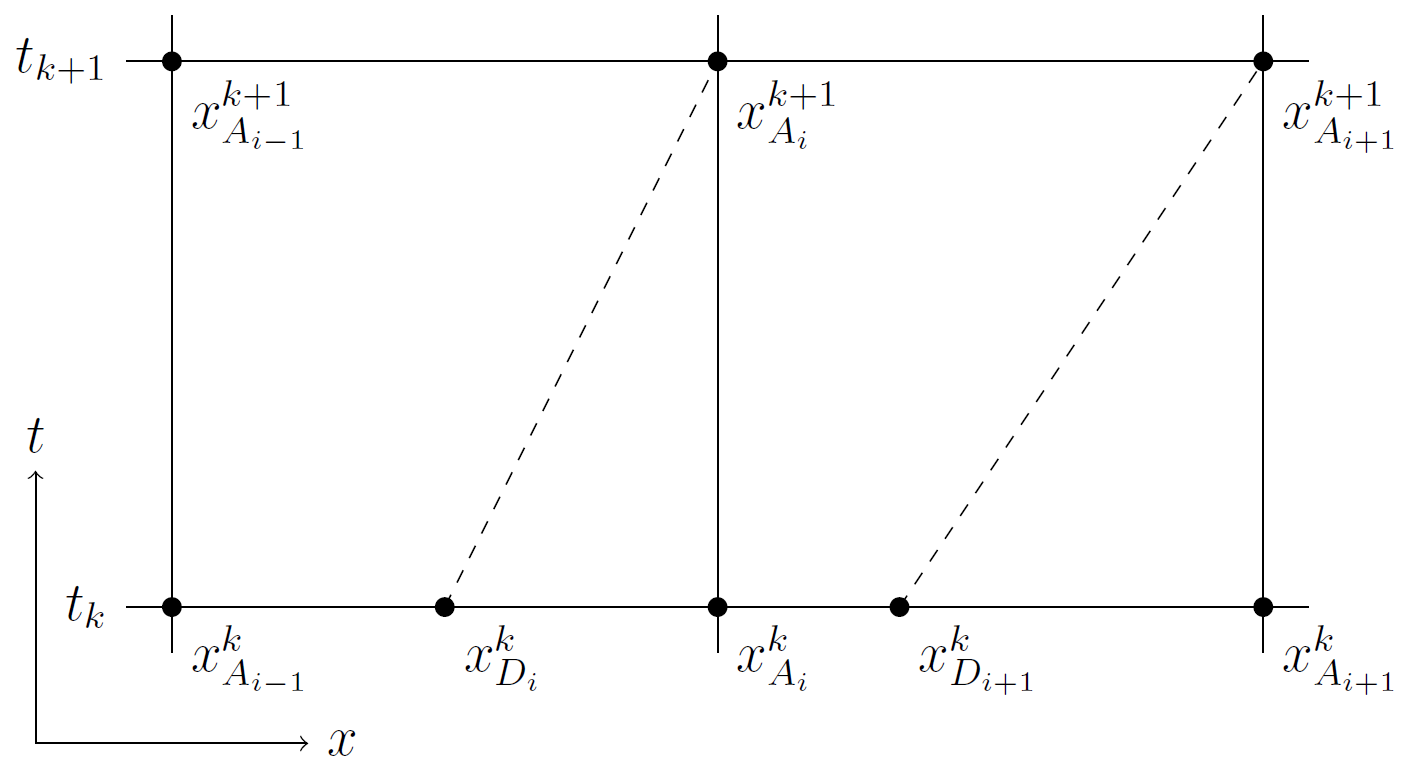}
\caption{A backward trace of trajectories from arrival points.}
\label{backwardpic}
\end{figure}

Figure \ref{backwardpic} demonstrates the concept in one spatial dimension: the $t,x$-plane. Each $x_A^{k+1}$ refers to the grid point a particle will arrive on at time $t_{k+1}$. Each $x_D^k$ refers to the departure point for the particle at time $t_k$. Departure points and arrival points are paired as they represent the position of the same particle at two different times. We note that for a backward trace, the calculated departure point may not necessarily be on a grid intersection.

Early semi-Lagrangian theory began with the development around a finite difference framework \cite{staniforthcote}. To understand the elementary theory, consider the following advection equation in 1D,
\begin{equation}
\frac{dF}{dt} = \frac{\partial F}{\partial t} + \frac{dx}{dt} \frac{\partial F}{\partial x} = 0,
\label{simple1D}
\end{equation}
where $$\frac{dx}{dt} = U(t,x)$$
determines our characteristic lines.
By using a central difference approximation for $F$ about the point $(t_k, x_A - \Delta x)$ along a characteristic line and plugging into Equation (\ref{simple1D}), we have the finite difference approach,
\begin{equation}
\frac{F(t_k + \Delta t, x_A) - F(t_k - \Delta t, x_A - 2 \Delta x)}{2 \Delta t} = 0.
\label{simple1Ddiff}
\end{equation}
Point $(t_k + \Delta t, x_A)$ is a regular point on the Eulerian grid and $t_k - \Delta t = t_{k-1}$. However, due to varying values of $\Delta x$, we rarely find $(t_k - \Delta t, x_A - 2 \Delta x)$ to be a grid point. This requires us to use interpolation to determine values of $F$ at such a point.
By using the above characteristic equation, we can approximate $\Delta x$ from the implicit formula,
$$\Delta x = U(t_k, x_A - \Delta x)\Delta t,$$
which may be solved via iteration.

We recall that most of the present semi-Lagrangian-based methods take advantage of second order error with respect to the temporal discretisation, $\Delta t = t_{k+1} - t_k$. Our proposed algorithms achieve not just first and second order error with respect to the temporal discretisation, but also third and fourth order error for two general classes of time-dependent partial differential systems.

The rest of this paper is organised into five more sections. Section 2 will deal with the construction and implementation of our semi-Lagrangian algorithms in the context of the 1D initial value problem from System (\ref{mainsystem1}). Section 3 will consider the stability and convergence of the methods from section 2. Section 4 will explain how these methods can be developed for higher dimensional initial value problems, such as System (\ref{mainsystem2}). Section 5 consists of numerical results for two test cases, one for System (\ref{mainsystem1}) and another for System (\ref{mainsystem2}). Section 6 contains our conclusions and considerations for future research.

\section{Semi-Lagrangian Methods}
Recall the general advection in one dimension problem from System (\ref{mainsystem1}). Many semi-Lagrangian methods are developed in \cite{guo1, guo2} to numerically solve a single PDE. We seek to adapt these methods to solve the system of interest.
Since $\omega = \frac{dx}{dt}$, we have
\begin{equation*}
\begin{array}{rll}
\int_{x_D^k}^{x_A^{k+1}} dx & = & \int_{t_k}^{t_{k+1}}\omega dt \vspace{2mm} \\
x_D^k & = & x_A^{k+1} - \int_{t_k}^{t_{k+1}}\omega dt \vspace{2mm} \\
& \approx & x_A^{k+1} - \omega \tau,
\end{array}
\end{equation*}
where $\tau = t_{k+1} - t_k = \frac{T}{N}$. Also, let $h = \frac{b-a}{M}$.  \vspace{2mm}

Let $\boldsymbol{y} = \boldsymbol{(} y_1, \dots, y_n \boldsymbol{)}^{\intercal}$ and $\boldsymbol{f} = \boldsymbol{(} f_1, \dots, f_n \boldsymbol{)}^{\intercal}$, then the reformulation of PDE System (\ref{mainsystem1}) is
\begin{equation*}
\boldsymbol{y}_t + \omega \boldsymbol{y}_x = \boldsymbol{f}(t, x, \boldsymbol{y}), \text{\hspace{3mm}}
0 \leq t \leq T, \text{\hspace{3mm}} a \leq x \leq b, \text{\hspace{3mm}} \|\boldsymbol{y}\|_{\infty} < \infty.
\end{equation*}
Since $\omega = \dfrac{dx}{dt}$, we also have, from the chain rule, $\dfrac{d}{dt} = \dfrac{\partial}{\partial t} + \omega \dfrac{\partial}{\partial x}$. This allows us to write System (\ref{mainsystem1}) in the compact form
\begin{equation}
\dfrac{d \boldsymbol{y}}{dt} = \boldsymbol{f}(t, x, \boldsymbol{y}), \text{\hspace{3mm}} \boldsymbol{y}(0, x)=\boldsymbol{y}_0(x), \text{\hspace{3mm}} (t,x) \in [0, T]\times[a, b], \text{\hspace{3mm}} \| \boldsymbol{y} \|_{\infty} < \infty.
\label{systemomega}
\end{equation}
Integrating from $t_k$ to $t_{k+1}$, we get
$$ \boldsymbol{y}(t_{k+1}, x_{A}^{k+1}) = \boldsymbol{y}(t_k, x_D^k) + \int_{t_k}^{t_{k+1}}\boldsymbol{f}(t, x, \boldsymbol{y}) dt.$$

Let $\boldsymbol{y}_A^{k+1}$ be the numerical approximation of $\boldsymbol{y}(t_{k+1}, x_A^{k+1})$ and $\boldsymbol{y}_D^k$ be the numerical approximation of $\boldsymbol{y}(t_k, x_D^k)$. We can now construct the semi-Lagrangian Euler method for finding $\boldsymbol{y}_A^{k+1}$ at all grid points $x_A^{k+1}$ given the values for $\boldsymbol{y}_A^k$ at all grid points $x_A^k$ as seen in Algorithm \ref{SLEM}.
\begin{algorithm}[H]
\caption{Semi-Lagrangian Euler Method}\label{SLEM}
\begin{algorithmic}[1]
\Procedure{SLEM: find all $\boldsymbol{y}_A^{k+1}$ given all $\boldsymbol{y}_A^k$}{}
\State set initial guesses for all $x_D^k$
\ForAll{$x_A^{k+1}$} \do

\For{$i=1$ to $n$} \Comment{Iterate to find $x_D^k$}
\State interpolate using $\boldsymbol{y}_A^k, x_A^k$ to find $\boldsymbol{y}_D^k$ \Comment{At least 1st order interpolation}
\State $\omega_D^k = \omega(t_k, x_D^k, \boldsymbol{y}_D^k)$
\State $x_D^k = x_A^{k+1} - \omega_D^k$
\EndFor
\State interpolate using $\boldsymbol{y}_A^k, x_A^k$ to find $\boldsymbol{y}_D^k$ \Comment{At least 1st order interpolation}
\State $\boldsymbol{y}_A^{k+1} = \boldsymbol{y}_D^k + \tau \boldsymbol{f}(t_k, x_D^k, \boldsymbol{y}_D^k)$
\EndFor
\State \textbf{return} all $\boldsymbol{y}_A^{k+1}$
\EndProcedure
\end{algorithmic}
\end{algorithm}

Lipscomb \cite{nlipscomb} shows that the semi-Lagrangian Euler method in Algorithm \ref{SLEM} has a local truncation error of $O(\tau)$. The integer $n$ is the number of implicit calculations used to approximate each departure point. While $n$ can vary, Williamson \& Olson \cite{williamsonolson} and Simmons \cite{simmons} have demonstrated in most cases that just a few iterations provide rapid convergence to the correct departure point, a common choice being $n=5$.

We will also consider the departure point calculations by case of $\omega$.
\begin{enumerate}
\item If $\omega \in \mathbb{R}$, then the calculation of the departure requires neither iteration nor interpolation.
\item If $\omega = \omega(t,x)$, then the calculation of the departure points is defined implicitly. We employ an iterative method; however, interpolation is not required to determine the departure points.
\item If $\omega = \omega (t,x,\boldsymbol{y})$, then the calculation of the departure points is again defined implicitly. However, $\omega_D^k$ is dependent on unknown values of $\boldsymbol{y}_D^k$ since the numerical solution is only known at the arrival (grid) points, $\boldsymbol{y}_A^k$. Therefore, an interpolation scheme is employed. Linear interpolation is sufficient for the semi-Lagrangian Euler method.
\end{enumerate}

The modified semi-Lagrangian approach is developed from a Runge-Kutta Order-2 method for solving systems of ODEs, much like the semi-Lagrangian Euler method was based on Euler's method for solving systems of ODEs. The development of Runge-Kutta methods can be found in many numerical analysis texts such as Burden and Faires \cite{burdenfaires}. The modified semi-Lagrangian Euler method is shown in Algorithm \ref{MSLEM}.
\begin{algorithm}[H]
\caption{Modified Semi-Lagrangian Euler Method}\label{MSLEM}
\begin{algorithmic}[1]
\Procedure{MSLEM: find all $\boldsymbol{y}_A^{k+1}$ given all $\boldsymbol{y}_A^k$}{}
\State set initial guesses for all $x_D^k$
\ForAll{$x_A^{k+1}$} \do

\For{$i=1$ to $n$} \Comment{Iterate to find $x_D^k$}
\State interpolate using $\boldsymbol{y}_A^k, x_A^k$ to find $\boldsymbol{y}_D^k$ \Comment{At least 2nd order interpolation}
\State $\omega_D^k = \omega(t_k, x_D^k, \boldsymbol{y}_D^k)$
\State $\boldsymbol{k}_1 = \boldsymbol{f}(t_k, x_D^k, \boldsymbol{y}_D^k)$
\State $\boldsymbol{k}_2 = \boldsymbol{f}(t_{k+1}, x_A^{k+1}, \boldsymbol{y}_D^k + \tau \boldsymbol{k}_1)$
\State $\boldsymbol{y}_A^{k+1} = \boldsymbol{y}_D^k + \frac{\tau}{2} (\boldsymbol{k}_1 + \boldsymbol{k}_2)$
\If{$i \leq n-1$}
\State $\omega_A^{k+1} = \omega (t_{k+1}, x_A^{k+1}, \boldsymbol{y}_A^{k+1})$
\State $x_D^k = x_A^{k+1} - \frac{\tau}{2} (\omega_D^k + \omega_A^{k+1})$
\EndIf
\EndFor
\EndFor
\State \textbf{return} all $\boldsymbol{y}_A^{k+1}$
\EndProcedure
\end{algorithmic}
\end{algorithm}

We notice that the algorithm utilises a higher order recalculation of the departure points. In the first and second order methods, we used an Euler approximation to numerically integrate along the characteristic lines. Algorithm (\ref{MSLEM}), on the other hand, uses a trapezoid rule approximation:
$$\int_{t_k}^{t_{k+1}} \omega dt \approx \dfrac{\tau}{2} (\omega_D^{k} + \omega_A^{k+1}),$$
where, under case 3 assumptions for $\omega$, $$\omega_D^k=\omega(t_k,x_D^k,\boldsymbol{y}_D^k)$$ and $$\omega_A^{k+1} = \omega(t_{k+1},x_A^{k+1},\boldsymbol{y}_A^{k+1}).$$ In order to determine $\omega_A^{k+1}$, we may need to evaluate $\boldsymbol{y}_A^{k+1}$ at all the arrival points. This is dependent on which case of $\omega$ we have to deal with. We will consider what elements of the algorithm can be omitted by case for $\omega$:
\begin{enumerate}
\item If $\omega \in \mathbb{R}$, then the calculation of the departure points reduces to a simple subtraction formula. Neither iteration nor interpolation is required for computing $x_D^k$.
\item If $\omega = \omega(t,x)$, then $\omega_A^{k+1}$ is not dependent on $\boldsymbol{y}_A^{k+1}$. Only implicit iteration will be required to compute $x_D^k$.
\item If $\omega = \omega(t,x,\boldsymbol{y})$, then $\omega_A^{k+1}$ is dependent on $\boldsymbol{y}_A^{k+1}$. Both implicit iteration and interpolation are required for calculating departure points. Second-order or higher interpolation is used for the modified semi-Lagrangian Euler method.
\end{enumerate}
The modified semi-Lagrangian Euler method has a local truncation error of $O(\tau^2)$ \cite{nlipscomb}.

Much like the modified semi-Lagrangian Euler method, we can derive third and fourth order Runge-Kutta methods for semi-Lagrangian schemes. The semi-Lagrangian Runge-Kutta method of order-3 is shown in Algorithm \ref{SLRKO3}.
\begin{algorithm}[H]
\caption{Semi-Lagrangian Runge-Kutta Order-3}\label{SLRKO3}
\begin{algorithmic}[1]
\Procedure{SLRKO3: find all $\boldsymbol{y}_A^{k+1}$ given all $\boldsymbol{y}_A^k$}{}
\State set initial guesses for all $x_D^k$
\ForAll{$x_A^{k+1}$} \do

\For{$i=1$ to $n$} \Comment{Iterate to find $x_D^k$}
\State interpolate using $\boldsymbol{y}_A^k, x_A^k$ to find $\boldsymbol{y}_D^k$ \Comment{At least 3rd order interpolation}
\State $\omega_D^k = \omega(t_k, x_D^k, \boldsymbol{y}_D^k)$
\State $\boldsymbol{k}_1 = \boldsymbol{f}(t_k, x_D^k, \boldsymbol{y}_D^k)$
\State $\boldsymbol{k}_2 = \boldsymbol{f}(t_k + \frac{\tau}{2}, x_D^k + \frac{\tau}{2} \omega_D^k, \boldsymbol{y}_D^k + \frac{\tau}{2} \boldsymbol{k}_1)$
\State $\boldsymbol{k}_3 = \boldsymbol{f}(t_{k+1}, x_A^{k+1}, \boldsymbol{y}_D^k - \tau \boldsymbol{k}_1 + 2 \tau \boldsymbol{k}_2)$
\State $\boldsymbol{y}_A^{k+1} = \boldsymbol{y}_D^k + \frac{\tau}{6} (\boldsymbol{k}_1 + 4 \boldsymbol{k}_2 + \boldsymbol{k}_3 )$
\If{$i \leq n-1$}
\State $\omega_A^{k+1} = \omega (t_{k+1}, x_A^{k+1}, \boldsymbol{y}_A^{k+1})$
\State $\tilde{\boldsymbol{k}}_2 = \boldsymbol{f}(t_k + \frac{\tau}{4}, x_D^k + \frac{\tau}{4} \omega_D^k, \boldsymbol{y}_D^k + \frac{\tau}{4} \boldsymbol{k}_1)$
\State $\tilde{\boldsymbol{k}}_3 = \boldsymbol{f}(t_{k} + \frac{\tau}{2}, x_D^k + \frac{\tau}{2} \omega_D^k, \boldsymbol{y}_D^k - \frac{\tau}{2} \boldsymbol{k}_1 + \tau \tilde{\boldsymbol{k}}_2)$
\State $\tilde{\boldsymbol{y}}_I = \boldsymbol{y}_D^k + \frac{\tau}{12} (\boldsymbol{k}_1 + 4 \tilde{\boldsymbol{k}}_2 + \tilde{\boldsymbol{k}}_3 )$
\State $\tilde{\omega}_I = \omega(t_k + \frac{\tau}{2}, x_D^k + \frac{\tau}{2} \omega_D^k, \tilde{\boldsymbol{y}}_I) $
\State $x_D^k = x_A^{k+1} - \frac{\tau}{6} (\omega_D^k + 4 \tilde{\omega}_I + \omega_A^{k+1})$
\EndIf
\EndFor
\EndFor
\State \textbf{return} all $\boldsymbol{y}_A^{k+1}$
\EndProcedure
\end{algorithmic}
\end{algorithm}

Algorithm \ref{SLRKO3} uses an even higher order recalculation of the departure points. In the first and second order methods, we used left-side-rule and trapezoid rule approximations to numerically integrate along the characteristic lines. Algorithm (\ref{SLRKO3}), on the other hand, uses a Simpson's rule approximation:
$$\int_{t_k}^{t_{k+1}} \omega dt \approx \dfrac{\tau}{6} (\omega_D^{k} +4 \tilde{\omega}_I + \omega_A^{k+1}),$$
where, under case 3 assumptions for $\omega$, $$\omega_D^k=\omega(t_k,x_D^k,\boldsymbol{y}_D^k),$$ $$\omega_A^{k+1} = \omega(t_{k+1},x_A^{k+1},\boldsymbol{y}_A^{k+1}),$$ and $$\tilde{\omega}_I = \omega \Big(t_{k+1}+\frac{\tau}{2}, x_D^k + \frac{\tau}{2} \omega_D^k, \tilde{\boldsymbol{y}}_I \Big).$$ In order to determine $\omega_A^{k+1}$, we may need to evaluate $\boldsymbol{y}_A^{k+1}$ at all the arrival points; further, a half-way Runge-Kutta computation is required for $\tilde{\omega}_I$. This is dependent on which case of $\omega$ we have to deal with. We will consider what elements of the algorithm can be omitted by case for $\omega$:
\begin{enumerate}
\item If $\omega \in \mathbb{R}$, then calculation of the departure points reduces to simple subtraction. Further, neither iteration nor interpolation is required for computing $x_D^k$.
\item If $\omega = \omega(t,x)$, then $\omega_A^{k+1}$ is not dependent on $\boldsymbol{y}_A^{k+1}$. Also, $\tilde{\omega}_I$ is not dependent on $\tilde{\boldsymbol{y}}_I$. Implicit iteration will be required to compute $x_D^k$.
\item If $\omega = \omega(t,x,\boldsymbol{y})$, then $\omega_A^{k+1}$ is dependent on $\boldsymbol{y}_A^{k+1}$. We also need to find $\tilde{\boldsymbol{y}}_I$ in order to determine $\tilde{\omega}_I$. Both implicit iteration and interpolation are required for calculating departure points. Third-order or higher interpolation is used for the semi-Lagrangian Runge-Kutta order-3 method.
\end{enumerate}

The semi-Lagrangian Runge-Kutta order-3 method has a local truncation error of $O(\tau^3)$ given that we use at least third-order interpolations within the method \cite{nlipscomb}.

We will now introduced the semi-Lagrangian Runge-Kutta method of order-4 as seen in Algorithm \ref{SLRKO4}.
\begin{algorithm}[H]
\caption{Semi-Lagrangian Runge-Kutta Order-4}\label{SLRKO4}
\begin{algorithmic}[1]
\Procedure{SLRKO4: find all $\boldsymbol{y}_A^{k+1}$ given all $\boldsymbol{y}_A^k$}{}
\State set initial guesses for all $x_D^k$
\ForAll{$x_A^{k+1}$} \do

\For{$i=1$ to $n$} \Comment{Iterate to find $x_D^k$}
\State interpolate using $\boldsymbol{y}_A^k, x_A^k$ to find $\boldsymbol{y}_D^k$ \Comment{At least 4th order interpolation}
\State $\omega_D^k = \omega(t_k, x_D^k, \boldsymbol{y}_D^k)$
\State $\boldsymbol{k}_1 = \boldsymbol{f}(t_k, x_D^k, \boldsymbol{y}_D^k)$
\State $\boldsymbol{k}_2 = \boldsymbol{f}(t_k + \frac{\tau}{2} , x_D^k + \frac{\tau}{2}  \omega_D^k, \boldsymbol{y}_D^k + \frac{\tau}{2} \boldsymbol{k}_1)$
\State $\boldsymbol{k}_3 = \boldsymbol{f}(t_k+\frac{\tau}{2}, x_D^k + \frac{\tau}{2} \omega_D^k, \boldsymbol{y}_D^k + \frac{\tau}{2} \boldsymbol{k}_2)$
\State $\boldsymbol{k}_4 = \boldsymbol{f}(t_{k+1}, x_A^{k+1}, \boldsymbol{y}_D^k + \tau \boldsymbol{k}_3)$
\State $\boldsymbol{y}_A^{k+1} = \boldsymbol{y}_D^k + \frac{\tau}{6} (\boldsymbol{k}_1 + 2 \boldsymbol{k}_2 +2 \boldsymbol{k}_3 + \boldsymbol{k}_4 )$
\If{$i \leq n-1$}
\State $\omega_A^{k+1} = \omega (t_{k+1}, x_A^{k+1}, \boldsymbol{y}_A^{k+1})$
\State $\tilde{\boldsymbol{k}}_2 = \boldsymbol{f}(t_k + \frac{\tau}{4} , x_D^k + \frac{\tau}{4}  \omega_D^k, \boldsymbol{y}_D^k + \frac{\tau}{4} \boldsymbol{k}_1)$
\State $\tilde{\boldsymbol{k}}_3 = \boldsymbol{f}(t_k+\frac{\tau}{4}, x_D^k + \frac{\tau}{4} \omega_D^k, \boldsymbol{y}_D^k + \frac{\tau}{4} \tilde{\boldsymbol{k}}_2)$
\State $\tilde{\boldsymbol{k}}_4 = \boldsymbol{f}(t_k + \frac{\tau}{2}, x_D^k + \frac{\tau}{2} \omega_D^k, \boldsymbol{y}_D^k + \tau \tilde{\boldsymbol{k}}_3)$
\State $\tilde{\boldsymbol{y}}_I = \boldsymbol{y}_D^k + \frac{\tau}{12} (\boldsymbol{k}_1 + 2 \tilde{\boldsymbol{k}}_2 +2 \tilde{\boldsymbol{k}}_3 + \tilde{\boldsymbol{k}}_4 )$
\State $\tilde{\omega}_I = \omega(t_k + \frac{\tau}{2}, x_D^k + \frac{\tau}{2} \omega_D^k, \tilde{\boldsymbol{y}}_I) $
\State $x_D^k = x_A^{k+1} - \frac{\tau}{6} (\omega_D^k + 4 \tilde{\omega}_I + \omega_A^{k+1})$
\EndIf
\EndFor
\EndFor
\State \textbf{return} all $\boldsymbol{y}_A^{k+1}$
\EndProcedure
\end{algorithmic}
\end{algorithm}

Once again, we are using a Simpson's rule approximation; therefore, an intermediate evaluation of $\tilde{\omega}_I$ is often needed. Depending on the case of $\omega$, this may require a half step-size evaluation of the solution, $\tilde{\boldsymbol{y}}_I$. The algorithmic alterations by case of $\omega$ are similar to the semi-Lagrangian order-3 method in Algorithm \ref{SLRKO3} with the exception that we use at least fourth order interpolation.

The semi-Lagrangian Runge-Kutta order-4 method has a local truncation error of $O(\tau^4)$ given that we use at least fourth order interpolations \cite{nlipscomb}.

\section{Stability and Convergence}

An important consideration when developing numerical algorithms for solving differential equations is whether the algorithm is numerically stable, convergent, and whether an upper bound on the error can be determined. We will consider this for the initial value problem posed by System (\ref{mainsystem1}). Before we introduce our stability and convergence theorem, we need a few lemmas. The first two lemmas and their proofs can be found in Burden and Faires \cite{burdenfaires}.

\begin{lemma}
For all $t \geq -1$ and any positive $m$, we have $0 \leq (1+t)^m \leq e^{mt}$.
\label{basiclemma1}
\end{lemma}
\begin{lemma}
If $s$ and $t$ are positive real numbers and $\{ a_k \}_{k=0}^n$ is a sequence satisfying $a_0 \geq \frac{-t}{s}$ and
$a_{k+1} \leq (1+s)a_k + t$, for each $k=0, 1, \dots, n-1$, then
$$ a_{k+1} \leq e^{(k+1)s} (a_0 + \frac{t}{s}) - \frac{t}{s}.$$
\label{basiclemma2}
\end{lemma}
We will also need an interpolation lemma from Guo \cite{guo2}.
\begin{lemma}
Suppose $x_i = a+ih$ for $i=0, 1, \dots, M$ and $h=(b-a)/M$. Then, for each $x \in [a,b]$, if $p(x)$ and $q(x)$ are two piece-wise linear interpolations with $p(x_i) = u_i$ and $q(x_i) = v_i$ for $i=0, 1, \dots, M,$ then
$$| p(x) - q(x) | \leq \max_{0 \leq i \leq M} |u_i - v_i|.$$
\label{interplemma}
\end{lemma}
Last, we need a theorem regarding the existence and uniqueness of solutions to initial value problems in the form of System (\ref{mainsystem1}).
\begin{theorem}
If $\boldsymbol{f}(t, x, \boldsymbol{y})$ from (\ref{mainsystem1}) is a continuous function of $t$ and $x$ and satisfies a Lipschitz condition in $\boldsymbol{y}$ for $(t, x) \in [0, T] \times [a, b]$ with $\| \boldsymbol{y} \|_{\infty} < \infty$, then there exists a unique differentiable solution, $\boldsymbol{y}(t, x)$ for the initial value problem in System (\ref{mainsystem1}).
\label{onedimexistunique}
\end{theorem}
A proof of Theorem \ref{onedimexistunique} can be derived from most textbooks in partial differential equations.
We will now move on to our stability and convergence theorem which applies to the initial value problem in System (\ref{mainsystem1}). It is a generalisation of the theorem in Guo \cite{guo2}.
\begin{theorem}
Let the initial value problem in (\ref{systemomega}) be approximated by the one-step difference method
\begin{equation}
\left\{ \begin{array}{lr}
         \boldsymbol{y}_A^0=\boldsymbol{y}_0(x_A^0), & x_A^0=x_A^k=x_i, \ \text{for} \ i=0, \dots,M, \vspace{2mm} \\
         x_D^k = x_A^{k+1} - \omega_D^k \tau, & \vspace{2mm} \\
        \boldsymbol{y}_A^{k+1}=\boldsymbol{y}_D^k+\tau\boldsymbol{\phi}(t_k, x_D^k, \boldsymbol{y}_D^k, \tau, h) & \text{for each} \ k=0, \dots, N-1. \end{array} \right.
\label{OSDM}
\end{equation}
If there exists $\tau_0, \alpha > 0$ and $\boldsymbol{\phi}(t, x, \boldsymbol{y}, \tau, h)$ is continuous and satisfies a Lipschitz condition in the variable $\boldsymbol{y}$ with Lipschitz constant $L$ on
$$D = \{ (t, x, \boldsymbol{y}, \tau, h) : (t, x) \in [0,T]\times[a,b], \ \|\boldsymbol{y}\|_{\infty} < \infty, \ 0 \leq \tau \leq \tau_0, \ 0 \leq h \leq \alpha \tau_0 \}.$$\\
Then,
\begin{enumerate}
\item The one-step method in (\ref{OSDM}) is stable.
\item If $\boldsymbol{\phi}$ additionally satisfies a Lipschitz condition in all of its variables, then the one-step method in (\ref{OSDM}) is convergent if and only if it is consistent; that is $\boldsymbol{\phi}(t, x, \boldsymbol{y}, 0, 0) = \boldsymbol{f}(t,x,\boldsymbol{y})$.
\item If a function $\varepsilon(\tau)$ exists and, for each $k=0, 1, \dots, N$, the local truncation error $\boldsymbol{\varepsilon}_k(\tau)$ satisfies $\|\boldsymbol{\varepsilon}_k(\tau)\|_{\infty} \leq \varepsilon(\tau)$ for $\tau \in [0, \tau_0]$, then
    $$
    max_{x_A} \| \boldsymbol{y}(t_k, x_A^k) - \boldsymbol{y}_A^k \|_{\infty} \leq e^{TL} \max_{x_A} \| \boldsymbol{y}_0(x_A) - \boldsymbol{y}_A^0 \|_{\infty} + \frac{\varepsilon(\tau)}{L}(e^{TL}-1)
    $$
    for $k=1, 2, \dots, N$.
\end{enumerate}
\label{OSDMstabilityconvergencetheorem}
\end{theorem}
\begin{proof}{(1.)}
Let $\{\boldsymbol{u}_A^k\}_{k=1}^N$ and $\{\boldsymbol{v}_A^k\}_{k=1}^N$ each satisfy the difference equation in (\ref{OSDM}) with initial conditions $\boldsymbol{u}_0$ and $\boldsymbol{v}_0$ respectively.
Let
$$\boldsymbol{E}_A^k = \boldsymbol{u}_A^k - \boldsymbol{v}_A^k$$ for point $(t_k, x_A^k)$ and
$$\boldsymbol{E}_D^k = \boldsymbol{u}_D^k - \boldsymbol{v}_D^k$$ for point $(t_k, x_D^k)$.
Also, let
$$\boldsymbol{E}^k = \max_x \|\boldsymbol{E}_A^k\|_{\infty}$$ on the regular grid points (arrival points).
Now,
$$\boldsymbol{u}_A^{k+1} - \boldsymbol{v}_A^{k+1} = \boldsymbol{u}_D^k - \boldsymbol{v}_D^k + \tau[\boldsymbol{\phi}(t_k, x_D^k, \boldsymbol{u}_D^k, \tau, h) - \boldsymbol{\phi}(t_k, x_D^k, \boldsymbol{v}_D^k, \tau, h)]$$
from the difference formula. So,
$$
\begin{array}{lll}
\| \boldsymbol{u}_A^{k+1} - \boldsymbol{v}_A^{k+1} \|_{\infty} & = & \| \boldsymbol{u}_D^{k} - \boldsymbol{v}_D^{k} + \tau [ \boldsymbol{\phi}(t_k, x_D^k, \boldsymbol{u}_D^k, \tau, h) -  \boldsymbol{\phi}(t_k, x_D^k, \boldsymbol{v}_D^k, \tau, h)] \|_{\infty} \vspace{2mm} \\
& \leq & \| \boldsymbol{u}_D^{k} - \boldsymbol{v}_D^{k} \|_{\infty} + \tau L \| \boldsymbol{u}_D^{k} - \boldsymbol{v}_D^{k} \|_{\infty} \vspace{2mm} \\
& = & (1+\tau L) \| \boldsymbol{u}_D^{k} - \boldsymbol{v}_D^{k} \|_{\infty}
\end{array}
$$
from our Lipschitz condition of $\boldsymbol{\phi}$. Thus,
$$ \| \boldsymbol{E}_A^{k+1} \|_{\infty} \leq (1+\tau L) \| \boldsymbol{E}_D^k \|_{\infty}.$$
By Lemma \ref{interplemma}, we have
$$ \| \boldsymbol{E}_A^{k+1} \|_{\infty} \leq (1+\tau L) \|\boldsymbol{E}^k \|_{\infty}$$
and
$$\| \boldsymbol{E}^{k+1} \|_{\infty} \leq (1+\tau L) \| \boldsymbol{E}^k \|_{\infty}.$$
From Lemmas \ref{basiclemma1} and \ref{basiclemma2},
$$
\begin{array}{lll}
\| \boldsymbol{E}^k \|_{\infty} & \leq & (1+\tau L)^k \| \boldsymbol{E}^0 \|_{\infty} \vspace{2mm} \\
& \leq & e^{k \tau L} \| \boldsymbol{E}^0 \|_{\infty} \vspace{2mm} \\
& \leq & e^{T L} \| \boldsymbol{E}^0 \|_{\infty}
\end{array}
$$
where $k \leq \frac{T}{\tau} = N$. We notice that $e^{TL}$ is our stability constant. So, the method in (\ref{OSDM}) is stable.
\end{proof}
\begin{proof}{(2.)}
Let $\boldsymbol{\phi}(t, x, \boldsymbol{y}, 0, 0) = \boldsymbol{g}(t, x, \boldsymbol{y})$. $\boldsymbol{g}$ satisfies the conditions in Theorem \ref{onedimexistunique}; so,
\begin{equation}
\left\{ \begin{array}{ll}
         \frac{d \boldsymbol{v} }{dt} = \boldsymbol{g}(t, x, \boldsymbol{v}), & (t, x) \in [0,T]\times[a,b],\\
         & \boldsymbol{v}(0, x) = \boldsymbol{y}_0(x), \end{array} \right.
\label{OSDMproof2eq}
\end{equation}
has a unique, differentiable solution $\boldsymbol{v}(t, x)$. The numerical solution $\boldsymbol{z}$ satisfies
\begin{equation}
\left\{ \begin{array}{ll}
         \boldsymbol{z}_A^0 = \boldsymbol{y}_0(x_A^0), & x_A^0 = x_i \ \text{for} \ i=0, \dots, M, \vspace{2mm} \\
         x_D^k = x_A^{k+1} - \tau \omega_D^k, & \vspace{2mm} \\
         \boldsymbol{z}_A^{k+1} = \boldsymbol{z}_D^k - \tau \boldsymbol{\phi}(t_k, x_D^k, \boldsymbol{z}_D^k, \tau, h), & \text{for} \ k=0, \dots, N-1. \end{array} \right.
\label{OSDMconvergencesystem}
\end{equation}
By the Mean Value Theorem,
$$\boldsymbol{v}(t_{k+1}, x_A^{k+1}) = \boldsymbol{v}(t_k, x_D^k) + \tau \boldsymbol{g}(t_k + \xi \tau, x_D^k + \eta h, \boldsymbol{v}(t_k + \xi \tau, x_D^k + \eta h))$$
for some $\xi, \eta \in [0, 1]$. Let
$$\boldsymbol{E}_A^{k+1} = \boldsymbol{z}_A^{k+1} - \boldsymbol{v}(t_{k+1}, x_A^{k+1}),$$
$$E^k = \max_{x_A} \| \boldsymbol{E}_A^k \|_{\infty}$$
on the regular grid points (arrival points), and
$$\boldsymbol{E}_D^k = \boldsymbol{z}_D^k - \boldsymbol{v}(t_k, x_D^k).$$
Then,
$$
\begin{array}{lll}
\boldsymbol{E}_A^{k+1} &=& \boldsymbol{E}_D^k + \tau[\boldsymbol{\phi}(t_k, x_D^k, \boldsymbol{z}_D^k, \tau, h) - \boldsymbol{g}(t_k + \xi \tau, x_D^k + \eta h, \boldsymbol{v}(t_k + \xi \tau, x_D^k + \eta h))] \vspace{2mm} \\
& = & \boldsymbol{E}_D^k +\tau[\boldsymbol{\phi}(t_k, x_D^k, \boldsymbol{z}_D^k, \tau, h) - \boldsymbol{\phi}(t_k, x_D^k, \boldsymbol{v}(t_k, x_D^k), \tau, h) \vspace{2mm} \\
&& + \boldsymbol{\phi}(t_k, x_D^k, \boldsymbol{v}(t_k, x_D^k), \tau, h) - \boldsymbol{\phi}(t_k, x_D^k, \boldsymbol{v}(t_k, x_D^k), \tau, 0) \vspace{2mm} \\
&& + \boldsymbol{\phi}(t_k, x_D^k, \boldsymbol{v}(t_k, x_D^k), \tau, 0) - \boldsymbol{\phi}(t_k, x_D^k, \boldsymbol{v}(t_k, x_D^k), 0, 0) \vspace{2mm} \\
&& + \boldsymbol{\phi}(t_k, x_D^k, \boldsymbol{v}(t_k, x_D^k), 0, 0) - \boldsymbol{g}(t_k+\xi\tau, x_D^k+\eta h, \boldsymbol{v}(t_k + \xi\tau, x_D^k + \eta h))].
\end{array}
$$
By the assumption of $\boldsymbol{\phi}$ satisfying Lipschitz in all of its variables, with Lipschitz constant $L$ that is sufficiently large enough to satisfy the Lipschitz conditions in all the variables, we have
$$ \| \boldsymbol{\phi}(t_k, x_D^k, \boldsymbol{z}_D^k, \tau, h) - \boldsymbol{\phi}(t_k, x_D^k, \boldsymbol{v}(t_k, x_D^k, \tau, h)) \|_{\infty} \leq L \| \boldsymbol{z}_D^k - \boldsymbol{v}(t_k, x_D^k) \|_{\infty} \leq L \| \boldsymbol{E}_D^k \|_{\infty},$$
$$ \| \boldsymbol{\phi}(t_k, x_D^k, \boldsymbol{v}(t_k, x_D^k), \tau, h) - \boldsymbol{\phi}(t_k, x_D^k, \boldsymbol{v}(t_k, x_D^k), \tau, 0) \|_{\infty} \leq Lh,$$
$$ \| \boldsymbol{\phi}(t_k, x_D^k, \boldsymbol{v}(t_k, x_D^k), \tau, 0) - \boldsymbol{\phi}(t_k, x_D^k, \boldsymbol{v}(t_k, x_D^k), 0, 0) \|_{\infty} \leq L\tau,$$
and
$$
\begin{array}{ll}
&\| \boldsymbol{\phi}(t_k, x_D^k, \boldsymbol{v}(t_k, x_D^k), 0, 0) - \boldsymbol{g}(t_k + \xi\tau, x_D^k + \eta h, \boldsymbol{v}(t_k + \xi\tau, x_D^k + \eta h)) \|_{\infty} \vspace{2mm} \\
\leq & \| \boldsymbol{g}(t_k, x_D^k, \boldsymbol{v}(t_k, x_D^k)) - \boldsymbol{g}(t_k+\xi\tau, x_D^k, \boldsymbol{v}(t_k, x_D^k)) \|_{\infty} \vspace{2mm} \\
& + \| \boldsymbol{g}(t_k+\xi\tau, x_D^k, \boldsymbol{v}(t_k, x_D^k)) - \boldsymbol{g}(t_k+\xi\tau, x_D^k+\eta h, \boldsymbol{v}(t_k, x_D^k)) \|_{\infty} \vspace{2mm} \\
& + \| \boldsymbol{g}(t_k+\xi\tau, x_D^k+\eta h, \boldsymbol{v}(t_k, x_D^k)) - \boldsymbol{g}(t_k+\xi\tau, x_D^k+\eta h, \boldsymbol{v}(t_k+\xi\tau, x_D^k+\eta h)) \|_{\infty} \vspace{2mm} \\
\leq & L \xi \tau + L \eta h + L \| \boldsymbol{v}(t_k, x_D^k) - \boldsymbol{v}(t_k +\xi \tau, x_D^k + \eta h) \|_{\infty} \vspace{2mm} \\
\leq & L \xi \tau + L \eta h + L L_1 \xi \tau + L L_1 \eta h \vspace{2mm} \\
\leq & L \tau + L h + L L_1 \tau + L L_1 h \vspace{2mm} \\
\leq & L(1+L_1)(\tau + h),
\end{array}
$$
where $L_1$ is a positive constant such that
$$\Big\| \frac{ \partial \boldsymbol{v}}{\partial t} \Big\|_{\infty} \leq L_1, \  \ \ \text{and} \ \ \ \ \
\Big\| \frac{ \partial \boldsymbol{v}}{\partial x} \Big\|_{\infty} \leq L_1$$
for all $(t, x) \in [0, T]\times[a,b]$. So,
$$
\begin{array}{lll}
E^{k+1} & \leq & \| \boldsymbol{E}_D^k \|_{\infty} + L \tau \| \boldsymbol{E}_D^k \|_{\infty} + L \tau h + L \tau^2 + L(1+L_1)\tau(\tau + h) \vspace{2mm} \\
&  \leq & (1+ \tau L)E^k + L (2+L_1) \tau (\tau + h)
\end{array}
$$
and
$$
E^{k+1} \leq (1+\tau L) E^k + L(2+L_1)\tau(\tau + H).
$$
Using Lemma \ref{basiclemma2},
$$
\begin{array}{lll}
E^k & \leq & e^{k\tau L} E^0  + (2+L_1)(\tau + h)(e^{k \tau L} -1) \ \ \text{and} \vspace{2mm} \\
E^k  & \leq & e^{TL} E^0 + (2+L_1)(\tau + h)(e^{TL} -1).
\end{array}
$$
Clearly, the right hand side of the inequality goes to zero as $\tau, h \to 0$ with $E^0 = 0$. So, $\boldsymbol{z}^k \to \boldsymbol{v}^k$; that is, our numerical solution converges to the solution of System (\ref{systemomega}). Thus, given $\boldsymbol{\phi}(t, x, \boldsymbol{y}, 0, 0) = \boldsymbol{g}(t, x, \boldsymbol{y})$, the one-step difference method in Equation (\ref{OSDM}) converges.

Now we will assume convergence of the method in Equation (\ref{OSDM}). $\boldsymbol{v}$, our unique solution to System (\ref{OSDMproof2eq}), matches $\boldsymbol{y}$, our unique solution to System (\ref{systemomega}). Let $\boldsymbol{\phi}$ and $\boldsymbol{g}$ differ at some point; now consider the initial value problem starting at that point. Obviously, the two solutions are different, which leads to a contradiction. Thus, convergence requires $\boldsymbol{\phi}(t, x, \boldsymbol{y}, 0, 0) = \boldsymbol{g}(t, x, \boldsymbol{y})$, the condition of consistency.
\end{proof}
\begin{proof}{(3.)}
Let $\boldsymbol{E}_A^k = \boldsymbol{y}_A^k - \boldsymbol{y}(t_k, x_A^k)$ and $E^k = \max_{x_A} \| \boldsymbol{E}_A^k \|_{\infty}$.
From the definition of local truncation error,
$$
\boldsymbol{y}(t_{k+1}, x_A^{k+1})  =  \boldsymbol{y}(t_k, x_D^k) + \tau \boldsymbol{\phi}(t_k, x_D^k, \boldsymbol{y}(t_k, x_D^k), \tau, h) + \tau \boldsymbol{\varepsilon}_{k+1}(\tau).
$$
Subtracting this from the difference method yields
$$\begin{array}{lll}
\boldsymbol{E}_A^{k+1} &=& \boldsymbol{E}_D^k + \tau [\boldsymbol{\phi}(t_k, x_D^k, \boldsymbol{y}_D^k, \tau, h) - \boldsymbol{\phi}(t_k, x_D^k, \boldsymbol{y}(t_k, x_D^k), \tau, h)] \vspace{2mm} \\
&& - \tau \boldsymbol{\varepsilon}_{k+1}(\tau).
\end{array}$$
Then,
$$ \| \boldsymbol{E}_A^{k+1} \|_{\infty} = \| \boldsymbol{E}_D^k \|_{\infty} + \tau L \| \boldsymbol{y}_D^k - \boldsymbol{y}(t_k, x_D^k) \|_{\infty} + \tau \| \boldsymbol{\varepsilon}_{k+1}(\tau) \|_{\infty}
$$
and
$$ E^{k+1} \leq (1+\tau L) E^k + \tau \varepsilon(\tau).$$
From Lemma \ref{basiclemma2},
$$ \begin{array}{rll}
E^k &\leq& e^{k\tau L}(E^0 +\frac{\varepsilon(\tau)}{L}) - \frac{\varepsilon(\tau)}{L}, \vspace{2mm} \\
\Rightarrow E^k &\leq& e^{TL} E^0 + \frac{\varepsilon(\tau)}{L}(e^{TL} -1),
\end{array}$$
for $k=1, 2, \dots, N$.
\end{proof}
This theorem can be adapted for the different methods by determining the correct form of $\boldsymbol{\phi}(t_k, x_D^k, \boldsymbol{y}_D^k, \tau, h)$, higher order interpolations, and also considering the higher order approximations of the departure points.

\section{Higher Dimensions}
We have also developed semi-Lagrangian methods for solving initial value problems in the form of nonlinear advection in two dimensions seen in System (\ref{mainsystem1}). For example, the semi-Lagrangian Euler method in two dimensions is seen in Algorithm \ref{SLEM2D}.
\begin{algorithm}[H]
\caption{Semi-Lagrangian Euler Method in Two Dimensions}\label{SLEM2D}
\begin{algorithmic}[1]
\Procedure{SLEM-2D: find all $\boldsymbol{u}_A^{k+1}$ given all $\boldsymbol{u}_A^k$}{}
\State set initial guesses for all $x_D^k, y_D^k$ pairs
\ForAll{$x_A^{k+1}, y_A^{k+1}$ pairs} \do

\For{$i=1$ to $n$} \Comment{Iterate to find $x_D^k$ and $y_D^k$}
\State interpolate using $\boldsymbol{u}_A^k, x_A^k, y_A^k$ to find $\boldsymbol{u}_D^k$ \Comment{At least 1st order interpolation}
\State $x_D^k = x_A^{k+1} - u_D^k$
\State $y_D^k = y_A^{k+1} - v_D^k$
\EndFor
\State interpolate using $\boldsymbol{u}_A^k, x_A^k, y_A^k$ to find $\boldsymbol{u}_D^k$ \Comment{At least 1st order interpolation}
\State $\boldsymbol{u}_A^{k+1} = \boldsymbol{u}_D^k + \tau \boldsymbol{f}(t_k, x_D^k, y_D^k, \boldsymbol{u}_D^k)$
\EndFor
\State \textbf{return} all $\boldsymbol{u}_A^{k+1}$
\EndProcedure
\end{algorithmic}
\end{algorithm}
$\boldsymbol{u}_A^k$ is the numerical approximation of $\boldsymbol{u}(t_k,x_A,y_A)=\boldsymbol{(} u(t_k,x_A,y_A),v(t_k,x_A,y_A) \boldsymbol{)}^{\intercal}$.
This method has a local truncation error of $O(\tau)$ for $h_x, h_y = O(\tau)$, $h_x = \frac{(b-a)}{M_x}$, $h_y = \frac{d-c}{M_y}$. We see that the method is similar to the one dimensional case. However, we have no longer have a generalised advection term and we must consider departure points with respect to an extra dimension. The 2nd--4th order methods are similarly constructed; further, similar results to Theorem \ref{OSDMstabilityconvergencetheorem} also hold for these algorithms as seen in Lipscomb \cite{nlipscomb}.

\section{Numerical Results}
We will now consider some numerical results. In order to do so, we require a measure of error.
\begin{definition}
Let $y_{\boldsymbol{x}}^t$ be the value of a numerical solution to a PDE at point $(t, \boldsymbol{x})$ on an Eulerian grid. Also, let $y(t, \boldsymbol{x})$ be the value of the exact solution to the same PDE at the same point. We define the residual at point $(t, \boldsymbol{x})$ as
$$
Res(y(t, \boldsymbol{x})) = |y(t, \boldsymbol{x}) - y_{\boldsymbol{x}}^t |.
$$
We further define the maximum residual at time $t$ as
$$
\max Res[y(t)] = \max_{\boldsymbol{x}} Res(y(t, \boldsymbol{x})).
$$
\end{definition}
Clearly, residuals capture the absolute error between the exact solution and numerical solution at specified points in time and space. We will use the maximum residual at time $t$ to create residual plots against time.

Our first initial value problem is in the form of general advection in one dimension (\ref{mainsystem1}). Specifically, we will work with $\omega(t,x,u,v) = u + v$, where $u = u(t,x)$ and $v = v(t,x)$ are the solutions to the system:
\begin{equation}
\begin{array}{lll}
u_t + u u_x + v u_x & = & 2 \pi (v^2 + uv -v) \vspace{1mm} \\
v_t + u v_x + v v_x & = & 2 \pi (u - u^2 - uv)
\end{array}
\label{problem1c}
\end{equation}
on the domain $(t,x) \in [0,1] \times [0,1]$ with initial conditions
$$
u(0, x) = \sin (2 \pi x), \ \ \ \ \ v(0, x) = \cos (2 \pi x).
$$
We can verify that the exact solutions are
$$
u(t, x) = \sin (2 \pi (x-t)), \ \ \ \ \ v(t, x) =  \cos (2 \pi (x-t)).
$$
In this case, the nonlinear system is specified strictly in terms of the unknown functions and their partial derivatives. We will use 5 iterations for calculating the departure points implicitly. However, we also note that interpolation is required for determining $\omega_D^k$ for each iteration.

The numerical solutions from the semi-Lagrangian Runge-Kutta Order-4 method can be seen in Figure \ref{RKO4-1c} using 50 spatial steps and 50 time steps; that is, $\tau = h = 0.02$. The solutions look very similar as they are simply translations of each other. However, they have different starting points and a full period is completed for both from $t=0$ to $t=1$.

\begin{figure}
\centering
\begin{tabular}{cc}
\includegraphics[width=70mm]{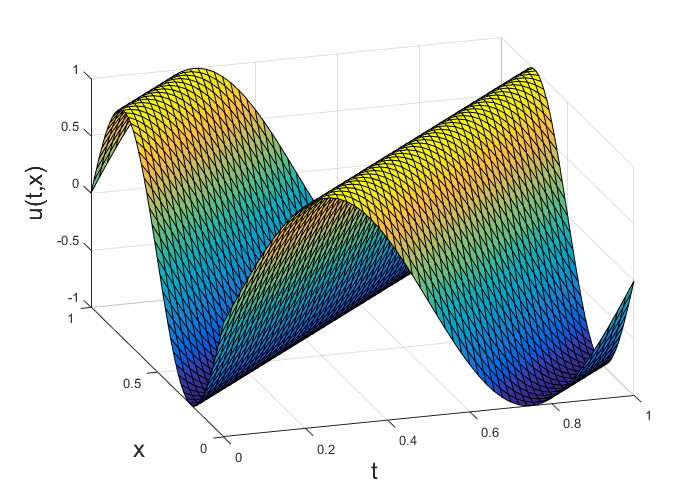} &
\includegraphics[width=70mm]{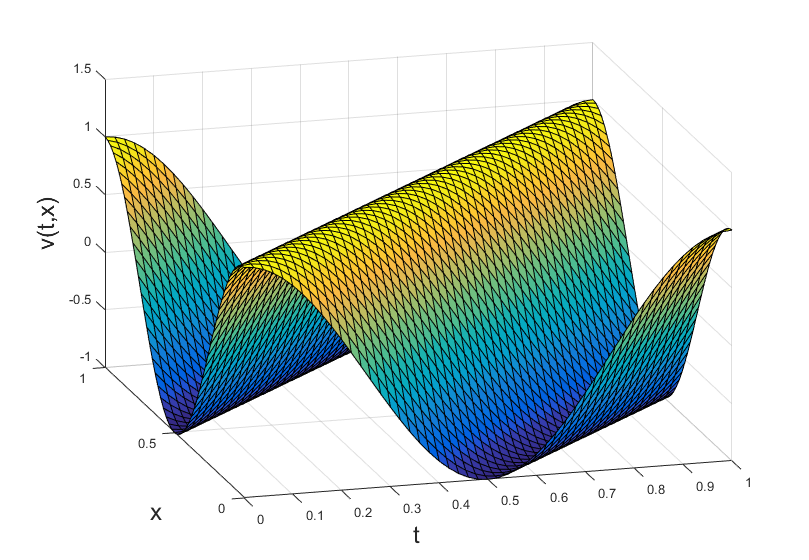}
\end{tabular}
\caption{Solutions to System (\ref{problem1c}) using the fourth order semi-Lagrangian method with $\tau = h = 0.02$.}
\label{RKO4-1c}
\end{figure}

Examining the semilog max residual plots in Figure \ref{maxresidual-1c}, it is clear we have achieved very strong results for the higher order methods. For all four methods, we chose $h = \tau = 0.005$.


\begin{figure}
\centering
\begin{tabular}{cc}
\includegraphics[width=87mm]{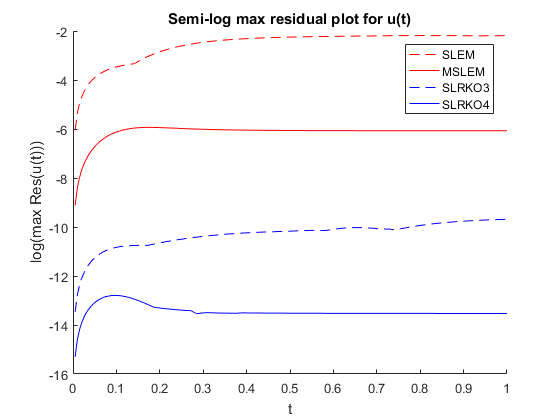} &
\includegraphics[width=87mm]{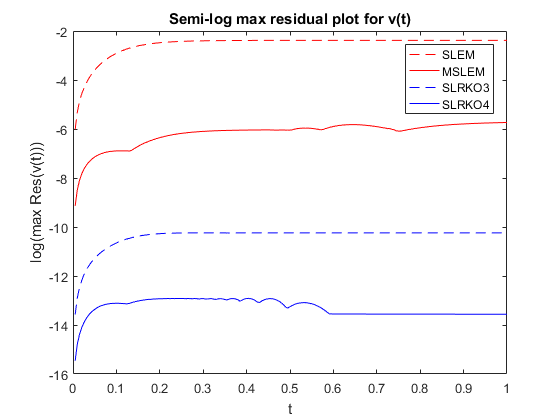}
\end{tabular}
\caption{Maximum residual semilog plot results from semi-Lagrangian method solutions to System (\ref{problem1c}) with $\tau = h = 0.005$.}
\label{maxresidual-1c}
\end{figure}

We will now consider a nonlinear advection in two dimensions (\ref{mainsystem2}) problem. Consider
\begin{equation}
\begin{array}{lll}
u_t + u u_x + v u_y & = & u + 2 \pi e^t (u \cos(2 \pi x) \sin(2 \pi y) + v \sin(2 \pi x) \cos(2 \pi y) \\
v_t + u v_x + v v_y & = & v - 2 \pi e^t (u \sin(2 \pi x) \cos(2 \pi y) + v \cos(2 \pi x) \sin(2 \pi y) \\
\end{array}
\label{problem2a}
\end{equation}
on the domain $(t,x,y) \in [0,1] \times [0,1] \times [0,1]$ with initial conditions
$$
u(0, x, y) = \sin(2 \pi x) \sin(2 \pi y), \ \ \ \ \ v(0, x, y) = \cos(2 \pi x) \cos(2 \pi y).
$$
We can verify that the exact solutions are
$$
u(t, x) = e^t \sin (2 \pi x) \sin (2 \pi y), \ \ \ \ \ v(t, x) = e^t \cos (2 \pi x) \cos (2 \pi y).
$$
The final time evolution of the numerical solutions from the fourth order method can be seen in Figure \ref{final-2a}. The step-sizes used were $\tau=h_x=h_y=0.02$. The initial conditions for both solutions are sine and cosine-based formations whose magnitude increases to what is seen in the picture.
\begin{figure}
\centering
\begin{tabular}{cc}
\includegraphics[width=87mm]{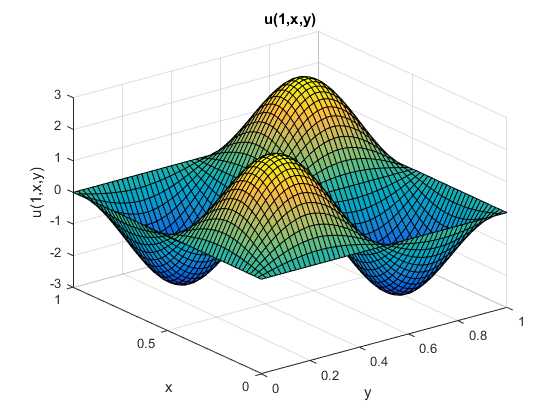} &
\includegraphics[width=87mm]{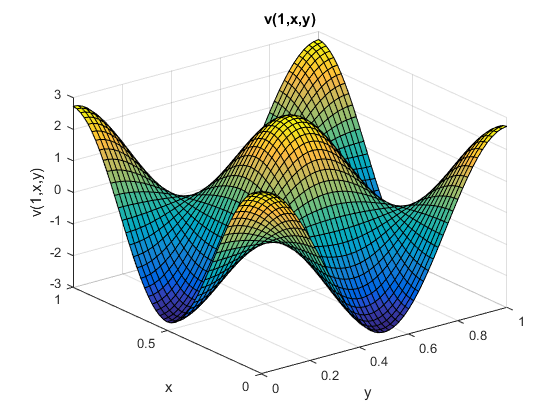}
\end{tabular}
\caption{The numerical solutions to $u$ and $v$ at $t=1$ for Problem (\ref{problem2a}) using a fourth order semi-Lagrangian method with $\tau=h_x=h_y=0.02$.}
\label{final-2a}
\end{figure}

Figure \ref{maxresidual-2a} presents the max semilog residual plots for comparative analysis. The spatial step-size was $h_x = h_y = 0.04$ with a time step-size of $\tau = 0.02$. Clearly, the fourth order method is superior, followed by the third order, second order, then first order.


\begin{figure}
\centering
\begin{tabular}{cc}
\includegraphics[width=87mm]{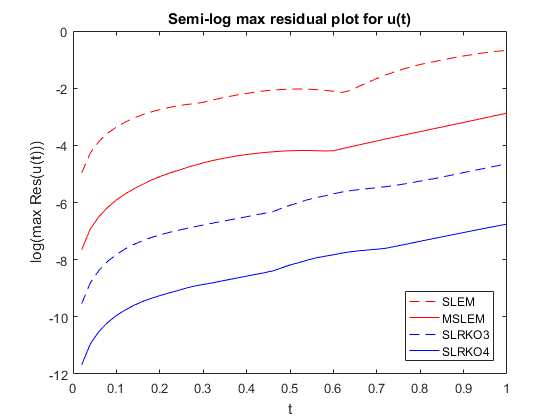} &
\includegraphics[width=87mm]{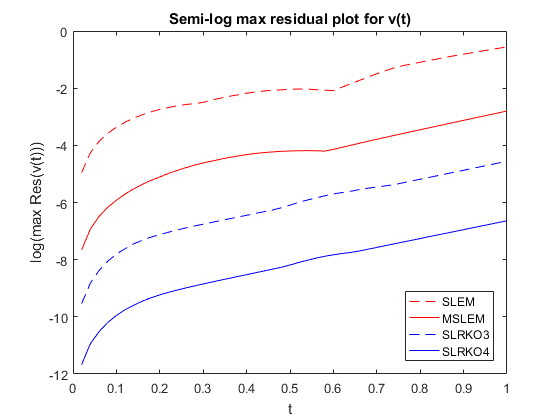}
\end{tabular}
\caption{Maximum residual semilog plot results from semi-Lagrangian method solutions to System (\ref{problem2a}) with $h_x = h_y = 0.04$ and $\tau = 0.02$.}
\label{maxresidual-2a}
\end{figure}

We will now provide some numerical confirmation of the order of these algorithms by examining the order of the absolute errors at $t=1$. The best way to examine the order of error is to produce a $\log$-$\log$ plot of the max residuals at $t=1$ as a function of the time step-size $\tau=\Delta t$. This will allow us to observe the slope to determine the order of the absolute error. Figure \ref{ConfofOrder} presents these plots for the first four algorithms in this paper; clearly the results match the order of each algorithm.

\begin{figure}
\centering
\begin{tabular}{cc}
\includegraphics[width=87mm]{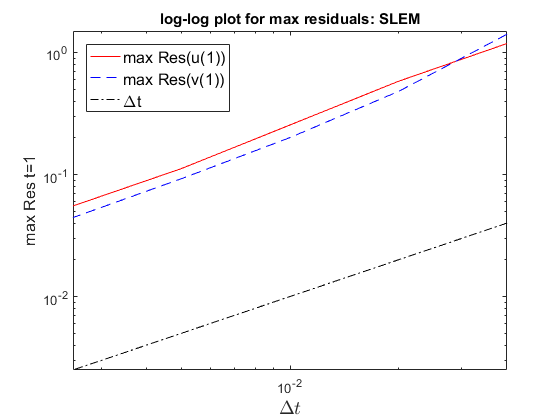} &
\includegraphics[width=87mm]{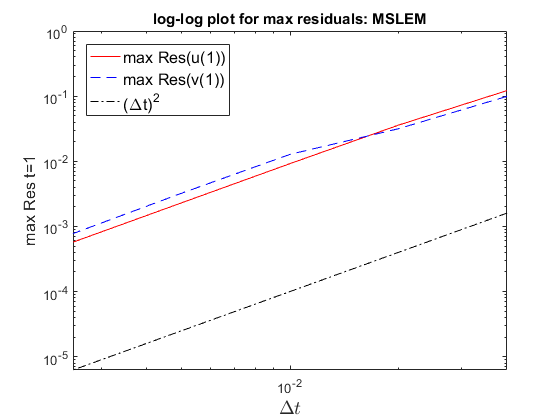} \\
\includegraphics[width=87mm]{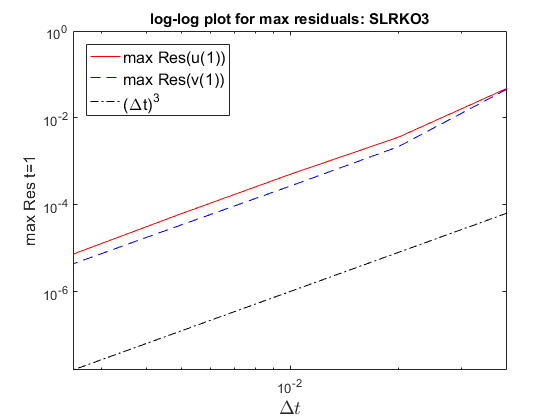} &
\includegraphics[width=87mm]{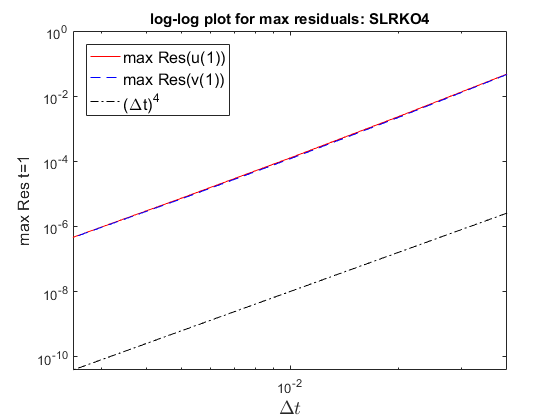}
\end{tabular}
\caption{$\log$-$\log$ plots relative to absolute error at $t=1$. Comparison lines have slopes to demonstrate exact order-$n$ change with respect to $\tau = \Delta t$; that is, the comparison lines have slopes of 1, 2, 3, and 4 respectfully.}
\label{ConfofOrder}
\end{figure}

\section{Conclusion and Future Development}

We have constructed several useful algorithms for numerically solving two general systems of time-dependent partial differential equations while examining their convergence, stability, and error. Using numerical simulations for several examples, we have established the effectiveness of the higher order numerical algorithms: the semi-Lagrangian Runge-Kutta order-3 and order-4 methods. The two general systems dealt with two different types of cases:
\begin{enumerate}
\item PDE System 1 is an initial value problem in one spatial dimension; however, the advection term $\omega$ can be generalised to $\omega = \omega(t,x,\boldsymbol{y})$ where $\boldsymbol{y}$ is the unknown solution to the initial value problem.
\item PDE System 2 is an initial value problem in two spatial dimensions. In this case, we have standard nonlinear advection terms: $u = \dfrac{dx}{dt}$ and $v = \dfrac{dy}{dt}$.
\end{enumerate}
By examining the construction of these algorithms, it is clear that these ideas can be built upon in order to solve PDE initial value problems in $n$ spatial dimensions as long as there are $n$ advection terms. These advection terms can be generalised as in PDE System 1; however, the performance of semi-Lagrangian methods in these more complicated cases will have to be tested.

\end{document}